\documentclass{amsart}

\usepackage{amsfonts,amssymb,amsmath,enumerate,verbatim,mathtools,tikz,bm,mathrsfs,tikz-cd,hyperref,comment,stmaryrd,amsthm}
\usepackage{colonequals}
\usepackage{adjustbox}
\usepackage[all,pdf]{xy}
\hypersetup{
    colorlinks=true,
    linkcolor=blue,
    filecolor=magenta,
    urlcolor=cyan,
}

\makeatletter
\@namedef{subjclassname@2020}{\textup{2020} Mathematics Subject Classification}
\makeatother

\author[Jian Liu]{Jian Liu}
\address{School of Mathematics and Statistics, Central China Normal University,  Wuhan 430079, P.R. China}
\email{jianliu@ccnu.edu.cn}

\author{Wei Ren}
\address{School of Mathematical Sciences, Chongqing Normal University, Chongqing 401331, P.R. China}
\email{wren@cqnu.edu.cn}

\keywords{ascent and descent, Gorenstein projective module, Gorenstein dimension, semi-dualizing complex, perfect complex.}

\subjclass[2020]{18G20 (primary); 16E65, 18G10, 18G80 (secondary)}

\DeclareMathOperator{\h}{H}

\newcommand{\n}{\mathfrak{n}}

\newcommand{\Z}{\mathbb{Z}}

\newcommand{\D}{\mathsf{D}}

\pgfdeclarelayer{bg}    
\pgfsetlayers{bg,main}

\newcommand{\m}{\mathfrak{m}}

\newcommand{\p}{\mathfrak{p}}

\DeclareMathOperator{\Ima}{Im}

\DeclareMathOperator{\pd}{pd}

\DeclareMathOperator{\Gdim}{G-dim}

\DeclareMathOperator{\id}{id}

\DeclareMathOperator{\Hom}{Hom}

\DeclareMathOperator{\Ext}{Ext}

\DeclareMathOperator{\RHom}{\mathsf{RHom}}

\newtheorem{theorem}{Theorem}[section]

\newtheorem{proposition}[theorem]{Proposition}

\newtheorem{lemma}[theorem]{Lemma}
\newtheorem{corollary}[theorem]{Corollary}

\theoremstyle{definition}
\newtheorem{example}[theorem]{Example}
\newtheorem{remark}[theorem]{Remark}

\newtheorem{definition}[theorem]{Definition}

\newtheorem{chunk}[theorem]{}

\usepackage[utf8]{inputenc}

\newtheorem*{ack}{Acknowledgements}

\title{Ascent and descent of Gorenstein homological properties}
\date{}

\begin{document}
\maketitle
\begin{abstract}
Let $\varphi\colon R\rightarrow A$ be a ring homomorphism, where $R$ is a commutative noetherian ring and $A$ is a finite $R$-algebra. We provide criteria for detecting the ascent and descent of Gorenstein homological properties.
Furthermore, we observe that the ascent and descent of Gorenstein homological property can detect the Gorenstein property of rings along  $\varphi$. 
\end{abstract}

\section{Introduction}

The study of the Gorenstein homological algebra can be traced back to the 1960s. Auslander and Bridger \cite{AB}  introduced the notion of the Gorenstein dimension for finitely generated modules, and they generalized the well-known Auslander-Buchsbaum formula to finitely generated modules with finite Gorenstein dimension. A finitely generated module with Gorenstein dimension zero is referred to as a Gorenstein projective module in \cite{EJ1}, and as a totally reflexive module in \cite{AM}. For an Iwanaga-Gorenstein ring, Buchweitz \cite{Buc} established a triangle equivalence between the stable category of finitely generated Gorenstein projective modules and the singularity category; this celebrated theorem highlighted the theory of the Gorenstein homological algebra.

For a surjective ring homomorphism $\varphi\colon R\rightarrow A$ of commutative noetherian local rings, if the kernel of $\varphi$ is generated by a regular sequence, then it is well-known that $A$ is Cohen-Macaulay (resp. Gorenstein, complete intersection) if and only if so is $R$.
Moreover, such ascent and descent properties along a ring homomorphism can be determined by certain homological properties; see for example \cite{AF, AFH, GS}.

Let $\varphi\colon R\rightarrow A$ be a ring homomorphism. We say $\varphi$ has \emph{ascent and descent of Gorenstein projective property} if each finitely generated left or right $A$-module is Gorenstein projective if and only if the underlying $R$-module is Gorenstein projective. An interesting example is due to Buchweitz \cite[8.2]{Buc}; he observed that for the integral group ring extension $\Z\rightarrow \Z G$ of a finite group $G$, a finitely generated $\mathbb{Z}G$-module (or equivalently, an integral representation of $G$) is Gorenstein projective if and only if the underlying $\mathbb{Z}$-module is Gorenstein projective (or equivalently, $\mathbb Z$-free). 

Notice that the above $\Z\rightarrow \Z G$ is a classical example of Frobenius extension \cite{Kas}. Inspired by this fact, Chen \cite{Chen} introduced the totally reflexive extension of rings and proved that such extension has ascent and descent of Gorenstein projective property. This motivates the subsequent works of Ren \cite{Ren} and Zhao \cite{Zhao} on the ascent and descent of Gorenstein projective property for (not necessarily finitely generated) modules along Frobenius extension of rings.

However, there are ring homomorphisms that satisfy the ascent and descent of Gorenstein projective property, but may not be Frobenius extensions; see Examples \ref{converseFro} and \ref{eg2}. Inspired by the aforementioned facts, it is natural to ask: how a ring homomorphism might behave if it has ascent and descent of Gorenstein projective property?

To study the ascent and descent of Gorenstein projective property, we first investigate in Section \ref{sectionadfd} the ring homomorphism $\varphi$ which has \emph{ascent and descent of finite Gorenstein dimension property}. The main result in this section is a characterization of this property; see Theorem \ref{observation}. 
Then, in Section \ref{sectionadgp}, we get the following result which concerns our question. 

\begin{theorem}\label{t2}(See \ref{ADTF})
Let $\varphi\colon R\rightarrow A$ be a ring homomorphism, where $R$ is a commutative noetherian ring and $A$ is a finite $R$-algebra.
The following two conditions are equivalent$\colon$
\begin{enumerate}
\item  $A$ is Gorenstein projective over $R$ and $\Hom_R(A,R)$ is projective as both a left and a right $A$-module.

\item $\varphi$ has ascent and descent of Gorenstein projective property.
\end{enumerate}
\end{theorem}
Let $\varphi$ be as in Theorem \ref{t2}. Recall that $\varphi\colon R\rightarrow A$ is a Frobenius extension provided that $A$ is a projective $R$-module and there is an isomorphism $\Hom_R(A,R)\cong A$ as left $A$-modules. According to \cite{Kad}, the second condition can be replaced by the existence of an isomorphism $\Hom_R(A,R)\cong A$ as right $A$-modules. It is of interest to compare condition (1) in the above with the condition for a Frobenius extension.

By making use of Theorem \ref{t2}, we show that the ascent and descent of Gorenstein projective property is a local property, i.e., $\varphi$ has ascent and descent of Gorenstein projective property if and only if $\varphi_\p\colon R_\p\rightarrow A_\p$ does so for each prime ideal (equivalently, maximal ideal) $\p$ of $R$; see Corollary \ref{local property}.


In Section \ref{testsection}, we study how the Gorenstein property of rings behaves along the ring homomorphism $\varphi\colon R\rightarrow A$. The main result in this section is the following. By combining it with Theorem \ref{observation} and a result of Christensen \cite[Proposition 8.3]{Chr01}, we can reobtain a result of Avramov and Foxby; see Corollary \ref{recover}.

\begin{theorem}\label{t3}(See \ref{ascentG})
Let $\varphi\colon R\rightarrow A$ be a ring homomorphism, where $R$ is a commutative noetherian ring and $A$ is a finite $R$-algebra. Assume $\varphi$ has ascent and descent of finite Gorenstein dimension property. If $R$ is an Iwanaga-Gorenstein ring, then so is $A$. The converse holds if, in addition, the fibre $A\otimes_R R/\m$ is nonzero for each maximal ideal $\m$ of $R$.
\end{theorem}

\section{Preliminaries}

We say a ring is noetherian if it is a two-sided noetherian ring.
 Throughout, $A$ will be a noetherian ring. In what follows,  a module over $A$ will mean a left $A$-module, and a module over $A^{\rm op}$  is identified with a right $A$-module. The category of $A$-modules and its full subcategory consisting of finitely generated $A$-modules will be denoted by $A\mbox{-Mod}$ and $A\mbox{-mod}$, respectively. An $A\mbox{-}A$-bimodule $M$ will mean a module over both $A$ and $A^{\rm op}$ that satisfies $(a_1\cdot m)\cdot a_2=a_1\cdot (m\cdot a_2)$ for all $a_1,a_2\in A$ and $m\in M$.

\begin{chunk}\label{gp}
\textbf{Gorenstein projective modules.} An acyclic complex of projective  $A$-modules
 $$
 \mathbf{P}  = \cdots \longrightarrow P_{1}\stackrel{\partial_{1}}\longrightarrow P_0\stackrel{\partial_0}\longrightarrow P_{-1}\longrightarrow\cdots
 $$
is called \emph{totally acyclic} provided that $\Hom_A(\mathbf{P}, Q)$ is still acyclic for any projective $A$-module $Q$. An $A$-module $M$ is \emph{Gorenstein projective} if there is a totally acyclic complex $\mathbf{P}$ such that $M$ is isomorphic to the image of $\partial_0$. Any projective  $A$-module is Gorenstein projective. For each totally acyclic complex $\mathbf{P}$, the image of $\partial_i$, denoted by ${\rm Im}(\partial_i)$, is Gorenstein projective for each $i\in \Z$.

The following characterization is well-known; see \cite[Proposition 3.8]{AB}. Let $M$ be a finitely generated  $A$-module. Then $M$ is Gorenstein projective if and only if $\Ext^{i}_A(M, A) = 0 = \Ext^{i}_{A^{\rm op}}(\Hom_A(M,A),A)=0$ for all $i>0$, and the evaluation homomorphism ${\rm e}_{M,A}\colon M\rightarrow \Hom_{A^{\rm op}}(\Hom_A(M,A),A)$ is an isomorphism. Thus, a finitely generated Gorenstein projective module is also called a {\em totally reflexive module}; see \cite{AM}.
\end{chunk}

\begin{chunk}
\textbf{Derived categories.} Let $\D(A)$ denote the derived category of complexes of  $A$-modules. It is a triangulated category with the suspension functor $[1]$; for each complex $X$, $X[1]_i\colonequals X_{i-1}$, and  $\partial_{X[1]}\colonequals -\partial_X$. Its full subcategory consisting of complexes with finitely generated total homology will be denoted by $\D^f_b(A)$. More precisely, for each $X\in \D(A)$, it is in $\D^f_b(A)$ if and only if $\h_i(X)$ is finitely generated for all $i$ and $\h_i(X)=0$ for all $|i|\gg 0$. The category $\D^f_b(A)$ inherits the structure of the triangulated category from $\D(A)$.

A complex $X$ is said to be \emph{homotopy projective} (resp. \emph{homotopy injective}) provided that $\Hom_A(X,-)$ (resp. $\Hom_A(-,X)$) preserves acyclic complexes. See \cite[Section 3]{Spa} for the existence of the homotopy projective resolution and the homotopy injective resolution of complexes. A complex $X$ is said to be \emph{bounded below} (resp. \emph{bounded above}) if $X_i=0$ for $i\ll 0$ (resp. $i\gg 0$). Every bounded below (resp. bounded above) complex of projective (resp. injective) modules is homotopy projective (resp. homotopy injective).

Let $\RHom_A(-,-)\colon \D(A)^{\rm op}\times \D(A)\rightarrow \D(\Z)$ denote the right derived functor of $\Hom_A(-,-)$. For each $M, N$ in $\D(A)$, the complex $\RHom_A(M, N)$ can be represented by either $\Hom_A(P,N)$ or $\Hom_A(M,I)$, where $P\xrightarrow \simeq M$ is a homotopy projective resolution and $N\xrightarrow\simeq I$ is a homotopy injective resolution.

Let $-\otimes_A^{\rm L}-\colon \D(A^{\rm op})\times \D(A)\rightarrow \D(\Z)$ denote the left derived functor of $-\otimes_A-$. For each $M$ in $\D(A^{\rm op})$ and $N$ in $\D(A)$, the complex $M\otimes_A^{\rm L}N$ can be represented by either $P\otimes_A N$ or $M\otimes_A Q$, where $P\xrightarrow \simeq M$ and $Q\xrightarrow \simeq N$ are homotopy projective resolutions over $A^{\rm op}$ and $A$, respectively.
\end{chunk}

\begin{chunk}
\textbf{Gorenstein dimensions.} Let $M$ be a complex in $\D^f_b(A)$. A quasi-isomorphism $G\xrightarrow \simeq M$ is a \emph{Gorenstein projective resolution} of $M$ provided that $G_i$ is a Gorenstein projective module for each $i\in \Z$ and  $G_{i}=0$ for $i\ll 0$. The \emph{Gorenstein dimension} of $M$, denoted by $\Gdim_A(M)$, is the smallest integer $n$ such that there is a Gorenstein projective resolution $G\xrightarrow \simeq M$ such that $G_i = 0$ for $i > n$ and $G_{n}\neq 0$; see for example \cite[Definition 2.3.2]{Chr00}.

For each $M\in A\mbox{-mod}$, we consider it as a stalk complex concentrated in degree zero, and $\Gdim_A(M)$ is precisely the Gorenstein projective dimension of $M$; see for example \cite[Definition 2.8]{Hol}. It is clear that $\Gdim_A(M)\leq \pd_A(M)$, where $\pd_A(M)$ is the projective dimension of $M$ over $A$; the equality holds if the latter is finite.
\end{chunk}

\begin{chunk}\label{def of I-Gorenstein}
\textbf{Iwanaga-Gorenstein rings.}
A noetherian ring $A$ is said to be \emph{Iwanaga-Gorenstein} provided that $A$ has finite injective dimension over both $A$ and $A^{\rm op}$. It follows from \cite[Lemma A]{Zaks} that if $A$ is Iwanaga-Gorenstein, then
$\id_A(A)=\id_{A^{\rm op}}(A)<\infty$. 

For each finitely generated $A$-module $M$, $\Gdim_A(M)\leq \id_A(A)$; see \cite[Corollary 2.21]{Hol}.
In particular, if $A$ is Iwanaga-Gorenstein, then each finitely generated $A$-module has finite Gorenstein dimension, and hence each complex in $\D^f_b(A)$ has finite Gorenstein dimension as well.
\end{chunk}

\begin{chunk}
\textbf{Perfect complexes.}  A complex $X$ in $\D(A)$ is said to be \emph{perfect} provided that it is isomorphic in $\D(A)$ to a bounded complex of finitely generated projective $A$-modules; equivalently, $X$ is compact as an object in $\D(A)$. See \cite[Chapter 1]{Buc} and \cite{Neeman} for more details about perfect complexes.

The full subcategory of $\D(A)$ consisting of perfect complexes is precisely the smallest triangulated subcategory of $\D(A)$ which contains $A$ and is closed under direct summands.
\end{chunk}
\begin{chunk}
\textbf{Semi-dualizing complexes.} Let $D$ be a complex of $A\mbox{-}A$-bimodules. $D$ is a \emph{semi-dualizing complex} over $A$ if the following conditions are satisfied:

(1) The total homology $\h(D)$ is finitely generated over $A$ and $A^{\rm op}$;

(2) $D_i=0$ for $i\gg 0$ and each $D_i$ is injective over $A$ and $A^{\rm op}$;

(3) The homothety morphisms
$$
{\rm m}_D\colon A\rightarrow \Hom_{A^{\rm op}}(D,D); a\mapsto (x\mapsto ax)
$$
and
$$
{\rm m}^\prime_D\colon A^{\rm {op}}\rightarrow \Hom_{A}(D,D); a\mapsto (x\mapsto xa)
$$
are quasi-isomorphisms. If, in addition, $D$ is a bounded complex of injective modules over $A$ and $A^{\rm op}$, then $D$ is called a \emph{dualizing complex}.

When $A$ is commutative, the above definition of the semi-dualizing complex coincides with the definition in \cite[Definition 2.1]{Chr01}.
\end{chunk}

\section{Ascent and descent of finite Gorenstein dimension property}\label{sectionadfd}
The main result of this section is Theorem \ref{observation} which provides a description of the ascent and descent of finite Gorenstein dimension property. 

For a ring homomorphism $\varphi\colon R\rightarrow A$, the map $\varphi$ is said to be \emph{finite} provided that $A$ is finitely generated over both $R$ and $R^{\rm op}$.

\begin{definition}\label{def:ADfGd}
Let $\varphi\colon R\rightarrow A$ be a finite ring homomorphism between noetherian rings.
We say $\varphi$ has \emph{ascent and descent of finite Gorenstein dimension property} if the following two conditions are satisfied$\colon$
\begin{enumerate}
\item For each complex in $\D^f_b(A)$, it has finite Gorenstein dimension over $A$  if and only if it has finite Gorenstein dimension over $R$;

\item For each complex in $\D^f_b(A^{\rm op})$, it has finite Gorenstein dimension over $A^{\rm op}$  if and only if it has finite Gorenstein dimension over $R^{\rm op}$.
\end{enumerate}
\end{definition}

By \ref{def of I-Gorenstein}, if $\varphi\colon R\rightarrow A$ is a finite ring homomorphism between Iwanaga-Gorenstein rings, then $\varphi$ has ascent and descent of finite Gorenstein dimension property.
Conversely, if $\varphi$ has this property, then the Iwanaga-Gorenstein property of rings is tested in Theorem \ref{ascentG}.

In the following, we abbreviate $\Hom_A(-,A)$ and $\Hom_{A^{\rm op}}(-,A)$ as $(-)^\ast$; there will be no confusion. Let $M$ be a complex in $\D^f_b(A)$. One can choose a homotopy projective resolution $\pi_M\colon P\xrightarrow \simeq M$, where $P$ is a bounded below complex of finitely generated projective  $A$-modules. Let $\iota\colon A\xrightarrow \simeq I$ be an injective resolution over $A^{\rm op}$. Denote by $\eta_M$ the composition of the following morphisms
$$
\xymatrix{
\eta_M\colon P \ar[r]^-{{\rm e}_{P,A}}_-\cong & \Hom_{A^{\rm op}}(P^\ast,A)\ar[r]^-{\iota_\ast}& \Hom_{A^{\rm op}}(P^\ast,I).
}
$$
Then, the \emph{biduality morphism} $\delta_M\colon M\rightarrow \RHom_{A^{\rm op}}(\RHom_A(M,A),A)$ of $M$, as a morphism in the derived category $\D(\Z)$, can be defined as the right fration
$$
\eta_M/\pi_M\colon M\rightarrow \Hom_{A^{\rm op}}(P^\ast,I).
$$

The following is a characterization of the finite Gorenstein dimension; see \cite[Theorem 10.4.5]{CFH}. In the commutative case, it is included in \cite[Corollary 2.3.8]{Chr00}.

\begin{lemma}\label{test}
Let $M$ be a complex in $\D^f_b(A)$. Then  $\Gdim_A(M)$ is finite if and only if the following two conditions are satisfied$\colon$
\begin{enumerate}
\item $\RHom_A(M,A)$ is in $\D^f_b(A^{\rm op})$;

\item The biduality morphism $\delta_M\colon M\rightarrow \RHom_{A^{\rm op}}(\RHom_A(M,A),A)$ is an isomorphism in $\D(\Z)$.
\end{enumerate}
\end{lemma}
\begin{remark}\label{dualfinite}
Let $M$ be a complex in $\D^f_b(A)$. If $\Gdim_A(M)$ is finite, then it follows from Lemma \ref{test} that $\Gdim_{A^{\rm op}}(\RHom_A(M,A))$ is also finite.

\end{remark}

\begin{lemma}\label{coincide}
Keep the same notations as above. Let $\pi\colon Q\xrightarrow\simeq P^\ast$ be a homotopy projective resolution of $P^\ast$ over $A^{\rm op}$. Then
\begin{enumerate}
\item The biduality morphism $\delta_M$ is an isomorphism in $\D(\Z)$ if and only if the map $\pi^\ast\colon P^{\ast\ast}\rightarrow Q^\ast$ is a quasi-isomorphism.

\item Let $X$ be a bounded above complex over $A$. If $X$ is perfect in $\D(A)$, then there is a quasi-isomorphism $$\pi\otimes X\colon Q\otimes_A X\xrightarrow\simeq  P^\ast\otimes_A X.$$
\end{enumerate}
\end{lemma}

\begin{proof}
(1) Let $\iota\colon A\xrightarrow \simeq I$ be an injective resolution over $A^{\rm op}$. Since $I$ is homotopy injective and $Q$ is homotopy projective, there is a commutative diagram
$$
\xymatrix{
P^{\ast\ast}\ar[r]^-{\pi^\ast}\ar[d]_-{\iota_\ast}& Q^\ast\ar[d]^-\simeq\\
\Hom_{A^{\rm op}}(P^\ast,I)\ar[r]^-\simeq & \Hom_{A^{\rm op}}(Q,I),
}
$$
where two unlabeled maps are induced by $\iota$ and $\pi$ respectively. Note that $\delta_M$ is an isomorphism in $\D(\Z)$ if and only if $\iota_\ast$ is a quasi-isomorphism. By the above diagram, this is equivalent to that $\pi^\ast$ is a quasi-isomorphism.

(2)
Since $X$ is perfect, by \cite[Theorem 5.1.14]{CFH} and  \cite[Lemma 1.2.1]{Buc}, there is a quasi-isomorphism $\alpha \colon F\xrightarrow \simeq X$  such that $F$ is a bounded complex of finitely generated projective  $A$-modules.
Consider the commutative diagram
$$
\xymatrix{
Q\otimes_A F\ar[r]^-{\pi\otimes F}\ar[d]_-{Q\otimes \alpha}& P^\ast\otimes_A F\ar[r]^-{\epsilon_{P,F}} \ar[d]_-{P^\ast\otimes \alpha} & \Hom_A(P,F)\ar[d]^-{\alpha_\ast}\\
Q\otimes_A X\ar[r]^-{\pi\otimes X} & P^\ast \otimes_A X\ar[r]^-{\epsilon_{P,X}}& \Hom_A(P,X),
}
$$
where $\epsilon_{P, F}$ and $\epsilon_{P,X}$ are the tensor evaluation isomorphisms; see \cite[Theorem 4.5.10]{CFH}.
Note that $\pi\otimes F$ and $Q\otimes \alpha$ are quasi-isomorphisms; see \cite[Proposition 5.8]{Spa}.  It is clear that $\alpha_\ast$ is a quasi-isomorphism. Therefore, the above diagram yields that $P^\ast\otimes \alpha$, and then $\pi\otimes X$, are quasi-isomorphisms.
\end{proof}

If $X$ is a bounded complex of finitely generated projective  $A$-modules,  the functor $-\otimes_A X$ preserves quasi-isomorphisms. However, this does not hold if $X$ is only assumed to be perfect, and therefore Lemma \ref{coincide} (2) is not trivial in general. For example, let $A=\mathbb Z/4\mathbb Z$ and $X=\cdots\xrightarrow {\overline{2}}  \mathbb Z/4\mathbb Z\xrightarrow {\overline{2}}   \mathbb Z/4\mathbb Z  \xrightarrow {\overline{2}}   \mathbb Z/4\mathbb Z \xrightarrow {\overline{2}}\cdots.$ Note that $X$ is acyclic, hence perfect, but $X\otimes_A X$ is not acyclic; indeed, the element of the form $(\cdots, 0,0, \overline{2},0,0,\cdots)$ in each degree of $X\otimes_A X$ is a cycle but not a boundary. This shows that $-\otimes_A X$ does not preserve quasi-isomorphisms.

In what follows, let $R$ be a commutative noetherian ring. The ring $A$ is said to be a {\em finite $R$-algebra} if there is a ring homomorphism $\varphi\colon R\rightarrow A$ such that the image of $\varphi$ is in the center of $A$ and $A$ is finitely generated as an $R$-module.

\begin{proposition}\label{theta}
Let $\varphi\colon R\rightarrow A$ be a ring homomorphism, where $A$ is a finite $R$-algebra and $\Gdim_R(A)$ is finite. Let $R\xrightarrow \simeq I$ be an injective resolution of $R$ and set $D=\Hom_R(A,I)$. Then
\begin{enumerate}
\item $D$ is a semi-dualizing complex over $A$.

\item If $D$ is a perfect complex of  $A$-modules, then for each complex $M\in \D^f_b(A)$,
there is an $A^{\rm op}$-linear quasi-isomorphism
$$\theta_M\colon  Q\otimes_A D\xrightarrow \simeq \Hom_R(P,I),$$
where $P$ is a bounded below complex of finitely generated projective $A$-modules such that $P\xrightarrow \simeq M$ is a homotopy projective resolution of $M$ over $A$, and $\pi\colon Q\xrightarrow \simeq \Hom_A(P,A)$ is a homotopy projective resolution of $\Hom_A(P,A)$ over $A^{\rm op}$.
\end{enumerate}
\end{proposition}

\begin{proof}
(1) For the commutative case, the result is established in \cite[Theorem 6.1]{Chr01}, and the argument extends naturally to the non-commutative setting.

(2) By the adjunction, there is an isomorphism
$$\Hom_A(P,D)= \Hom_A(P, \Hom_R(A, I)) \cong \Hom_R(P, I).$$
Consider the following  $A^{\rm op}$-linear morphisms
$$ Q\otimes_A D  \xrightarrow {\pi\otimes D} \Hom_A(P,A)\otimes_A D\xrightarrow {\epsilon_{P,D}}\Hom_A(P,D),$$
where $\epsilon_{P,D}$ is an isomorphism by the tensor evaluation; see \cite[Theorem 4.5.10]{CFH}. Hence, the required quasi-isomorphism follows immediately from Lemma \ref{coincide}.
\end{proof}

 If, in addition, $A$ is a commutative noetherian local ring, then the assumption of Proposition \ref{theta} (2) implies that $D$ is isomorphic to a suspension of $A$ in $\D^f_b(A)$; see \cite[Proposition 8.3]{Chr01}. However, in general, this is not true; see Example \ref{non-suspension}.
\begin{example}\label{non-suspension}
   Consider the ring homomorphism $\varphi\colon R=\mathbb Z\rightarrow A=  \mathbb Z/2\mathbb Z\times \mathbb Z $ defined by $\varphi(n)=(\overline{n},n)$. Set $D=\Hom_R(A,I)$ as in Proposition \ref{theta}. Since $A$ is regular, $D$ is a perfect complex over $A$. A direct calculation shows that $D\cong \mathbb Z/2\mathbb Z[-1]\oplus \mathbb Z$ in $\D^f_b(A)$. Consequently, $D$ is not isomorphic to a suspension of $A$ in $\D^f_b(A)$.
   \end{example}
\begin{example}\label{eg.dual}
Let $R$ be a commutative noetherian ring of characteristic $2$ and $A=R[x]/(x^2)$. Let $R\xrightarrow \simeq I$ be an injective resolution over $R$. By Proposition \ref{theta}, $\Hom_R(A,I)$ is a semi-dualizing complex over $A$. Moreover, $\Hom_R(A,I)$ is perfect over $A$ since there is an $A$-linear quasi-isomorphism $\Hom_R(A,R)\xrightarrow \simeq \Hom_R(A,I)$ and an isomorphism of $A$-modules
$A\xrightarrow \cong \Hom_R(A,R)$; note that $A$ is isomorphic to the group algebra $RG$ ( $\overline{x}\mapsto g+1$, $G$ is a group $\{1, g\mid g^2=1\}$ of two elements), the isomorphism $A\xrightarrow \cong \Hom_R(A,R)$ now follows from the same argument of \cite[Proposition 4.2.6]{Weibel}.

However, $\Hom_R(A,I)$ cannot be isomorphic to a dualizing complex in $\D^f_b(A)$ if $R$ is not Iwanaga-Gorenstein. Indeed, if $R$ is not Iwanaga-Gorenstein, then Theorem \ref{observation} and Theorem \ref{ascentG} will imply that neither is $A$. This yields that the injective resolution of $A$ cannot be bounded. Since $\Hom_R(A,I)$ is an injective resolution of $A$, we get that $\Hom_R(A,I)$ is not dualizing.
\end{example}

\begin{proposition}\label{quasi-iso}
Let $A$ be a noetherian ring and $D$ be a semi-dualizing complex. Assume $D$ is a perfect complex of  $A$-modules.
For each $M$ in $\D^f_b(A)$ with finite Gorenstein dimension, the evaluation morphism
$$
{\rm e}_{M,D} \colon M\rightarrow \Hom_{A^{{}\rm op}}(\Hom_A(M,D),D)
$$
is an $A$-linear quasi-isomorphism.
\end{proposition}

\begin{proof}
One can check directly that ${\rm e}_{M,D}$ is  $A$-linear. Note that $D$ is homotopy injective over $A$ and $A^{\rm op}$.
By taking a projective resolution of $M$, we may assume $M=P$ to be a bounded below complex of finitely generated projective  $A$-modules with finite Gorenstein dimension.

Keep the notations as above. Choose a homotopy projective resolution $\pi\colon Q\xrightarrow \simeq P^\ast$. Since $\Gdim_A(P)<\infty$, the biduality morphism $\delta_P$ is a quasi-isomorphism. It follows from Lemma \ref{coincide} that $\pi^\ast\colon P^{\ast\ast}\rightarrow Q^\ast$ is a quasi-isomorphism. Consider the commutative diagram
$$  \xymatrix{
    P\ar[rr]^-{{\rm e}_{P,D}}\ar[d]^-\cong_-{{\rm e}_{P,A}}& & \Hom_{A^{\rm op}}(\Hom_A(P,D),D)\ar[d]^-{(\epsilon_{P,D})^\ast}\\
    P^{\ast\ast}\ar[dd]_-{\pi^\ast}^-\simeq && \Hom_{A^{\rm op}}(P^\ast\otimes_A D,D)\ar[d]^-{(\pi\otimes D)^\ast}\\
  && \Hom_{A^{\rm op}}(Q\otimes_AD,D)\ar[d]^-\gamma_-\cong\\
    \Hom_{A^{\rm op}}(Q,A)\ar[rr]^-{({\rm m}_D)_\ast}& &\Hom_{A^{\rm op}}(Q,\Hom_{A^{\rm op}}(D,D)),
    }
$$
where the quasi-isomorphism $\pi^\ast$ is from Lemma \ref{coincide}, and $\gamma$ is induced from the adjunction. Since $D$ is homotopy injective and perfect over $A$, it follows from Lemma \ref{coincide} that $(\pi\otimes D)^\ast$ is a quasi-isomorphism, and we infer the isomorphism $({\rm e}_{P,D})^\ast$ by the tensor evaluation; see \cite[Theorem 4.5.10]{CFH}. By assumption, ${\rm m}_D$ is a quasi-isomorphism. Then so is $({\rm m}_D)_\ast$ since $Q$ is homotopy projective. Therefore, the above diagram yields that ${\rm e}_{P,D}$ is a quasi-isomorphism.
\end{proof}

\begin{remark}
If $A$ is a noetherian ring with a dualizing complex $D$, then there is a quasi-isomorphism
${\rm e}_{M,D} \colon M\xrightarrow \simeq \Hom_{A^{\rm op}}(\Hom_A(M,D),D)$ for any complex $M$ in $\D(A)$ with degree-wise finitely generated homology modules; see \cite[Theorem 10.1.23]{CFH}, the commutative case of this result is included in \cite[Chapter V, Proposition 2.1]{Hartshorne}.

\end{remark}




\begin{lemma}\label{projective}
Let $\varphi\colon R\rightarrow A$ be a ring homomorphism, where $R$ is a commutative noetherian ring and $A$ is a finite $R$-algebra. For a finitely generated Gorenstein projective  $A$-module $N$ and a prime ideal $\p$ of $R$, if there exists $n>0$ such that $\Ext^n_{A_\p}(M_\p, N_\p)=0$ for each finitely generated Gorenstein projective  $A$-module $M$, then $N_\p$ is projective over $A_\p$.
\end{lemma}

\begin{proof}
For $N$, there exists a long exact sequence
$$
 0\rightarrow N\rightarrow P_0\xrightarrow {\partial_0}P_{-1}\rightarrow \cdots\xrightarrow{\partial_{-(n-2)}} P_{-(n-1)}\rightarrow C\rightarrow 0,
$$
where $P_i$ is a finitely generated projective  $A$-module for each $i$, and $C$ is a finitely generated Gorenstein projective   $A$-module.

If $n=1$, then the short exact sequence $0\rightarrow N_\p\rightarrow (P_0)_\p\rightarrow C_\p\rightarrow 0$ splits, and $N_\p$ is projective over $A_\p$.
Assume $n>1$.  Applying $\Hom_A(C,-)$ to the above long exact sequence, we have
$$\Ext^n_A(C,N)\cong \Ext^{n-1}_A(C,\Ima(\partial_0))\cong \cdots\cong \Ext^1_A(C,\Ima (\partial_{-(n-2)})).$$
Combining with $\Ext^n_{A_\p}(C_\p, N_\p)=0$, we have $\Ext^1_{A_\p}(C_\p,\Ima (\partial_{-(n-2)})_\p)=0$. 
It follows that $C_\p$ is projective over $A_\p$, and hence all $\Ima (\partial_i)_\p$ and $N_\p$ are projective over $A_\p$.
\end{proof}


Now, we are in a position to state the main result of this section.

\begin{theorem}\label{observation}
Let $\varphi\colon R\rightarrow A$ be a ring homomorphism, where $R$ is a commutative noetherian ring and $A$ is a finite $R$-algebra.
The following two conditions are equivalent$\colon$
\begin{enumerate}
\item\label{basic}  $\Gdim_R(A)<\infty$ and $\RHom_R(A,R)$ is perfect over both $A$ and $A^{\rm op}$.

\item\label{TF} $\varphi$ has ascent and descent of finite Gorenstein dimension property.
\end{enumerate}
\end{theorem}

\begin{proof}
Set $D=\RHom_R(A,R)$, which can be represented by $\Hom_R(A,I)$, where $I$ is an injective resolution of $R$. In what follows, we identify $\RHom_R(-,R)$ with $\Hom_R(-,I)$.

$(\ref{basic})\Rightarrow (\ref{TF})$. Let $M$ be a complex in $\D^f_b(A)$. In the following, we will show that $M$ has finite Gorenstein projective dimension over $A$ if and only if it has finite Gorenstein projective dimension over $R$. The similar result holds for any complex in $\D^f_b(A^{\rm op})$.

First, assume $\Gdim_A(M)<\infty$. Since $D$ is perfect over $A$, it follows from Lemma \ref{test} that $\RHom_A(M,D)$ has finitely many non-zero homology modules. Then, we infer that $\RHom_R(M,R)$ is in $\D^f_b(R)$ by $\RHom_A(M,D)\cong \RHom_R(M,R)$. Consider the commutative diagram
$$
\xymatrix{
M\ar[r]^-{{\rm e}_{M,I}} \ar[d]_-{{\rm e}_{M,D}} & \Hom_R(\Hom_R(M,I),I)\ar[d]^\cong\\
\Hom_{A^{\rm op}}(\Hom_A(M,D),D)\ar[r]^-\cong & \Hom_{A^{{\rm op}}}(\Hom_R(M,I),D),
}
$$
where the unlabeled isomorphisms are due to  the adjunction $({\rm Res}, \Hom_R(A,-))$.
It follows from Propositions \ref{theta} and \ref{quasi-iso} that ${\rm e}_{M,D}$ is a quasi-isomorphism, and hence so is ${\rm e}_{M,I}$. Combining with $\RHom_R(M,R)\in \D^f_b(R)$,  we conclude by Lemma \ref{test} that $\Gdim_R(M)<\infty$. 

Now, assume $\Gdim_R(M)<\infty$. Since $D$ is a perfect complex over $A^{\rm op}$, one has $\Gdim_R(D\otimes_A ^{\rm L} M)<\infty$. Then,
by Lemma \ref{test} we infer that
$$\RHom_A(M,A)\simeq \RHom_A(M,\RHom_R(D,R)) \simeq\RHom_R(D\otimes_A^{\rm L} M,R)$$
is in $\D^f_b(A^{\rm op})$. Note that the first quasi-isomorphism holds since
$$
{\rm e}_{A,I}\colon A\xrightarrow {\simeq}\Hom_R(\Hom_R(A,I),I)= \RHom_{R}(D,R),
$$
which is due to $\Gdim_R(A)<\infty$, and is also linear over $A$ and $A^{\rm op}$.

For complex $M\in\D^f_b(A)$, there is a quasi-isomorphism $P\xrightarrow \simeq M$, where $P$ is a bounded below complex of finitely generated projective $A$-modules. Let $\pi\colon Q\xrightarrow \simeq \Hom_A(P,A)$ be a homotopy projective resolution. Consider the commutative diagram
$$
\xymatrix{
P\ar[rr]^-{{\rm e}_{P,A}}_-\cong \ar[d]_-{{\rm e}_{P,I}}^-\simeq && \Hom_{A^{\rm op}}(\Hom_A(P,A),A)\ar[d]^-{\pi^\ast}\\
\Hom_R(\Hom_R(P,I),I)\ar[d]_-{(\theta_M)^\ast} & &\Hom_{A^{\rm op}}(Q,A)\ar[d]^-{({\rm e}_{A,I})_\ast} \\
 \Hom_R(Q\otimes_A D,I)\ar[rr]^-\cong && \Hom_{A^{\rm op}}(Q,\Hom_R(D,I)),
}
$$
where the quasi-isomorphism ${\rm e}_{P,I}$ is from Lemma \ref{test}, the morphism $\theta_M$ is from Proposition \ref{theta}, and the unlabeled isomorphism is due to the adjunction $(-\otimes_A D, \Hom_R(D,-))$. Since $Q$ is homotopy projective and ${\rm e}_{A,I}$ is a quasi-isomorphism, one gets that $({\rm e}_{A,I})_\ast$ is also a quasi-isomorphism. It follows immediately from Proposition \ref{theta} that $(\theta_M)^\ast$ is a quasi-isomorphism as $I$ is homotopy injective. Thus, we conclude from the above diagram that $\pi^\ast$ is a quasi-isomorphism. Note that $\RHom_A(M,A)\in \D^f_b(A^{\rm op})$. Hence, $\Gdim_A(M)<\infty$ by Lemmas \ref{test} and \ref{coincide}.

$(\ref{TF})\Rightarrow (\ref{basic})$. Since $\varphi$ has descent of finite Gorenstein dimension property, we have $\Gdim_R(A)<\infty$. This yields that $\Gdim_R(D)<\infty$; see Remark \ref{dualfinite}. Since $\varphi$ has ascent of finite Gorenstein dimension property, $D$ has finite Gorenstein dimension over both $A$ and $A^{\rm op}$. Next, we will show that $D$ is perfect as a complex of  $A$-modules; the same argument holds for $D$ as a complex of $A^{\rm op}$-modules.

Let $\p$ be a prime ideal of $R$, and let $A\text{\rm -Gproj}$ denote the category of finitely generated Gorenstein projective  $A$-modules.  There exists an exact triangle in $\D^f_b(A)$:
\begin{equation}\label{cones}
K[n{\rm -}1]\rightarrow P\rightarrow D\rightarrow K[n], \tag{$\dagger$}
\end{equation}
where $n = \Gdim_{A}(D)$, $P$ is a perfect complex over $A$, and $K$ is finitely generated Gorenstein projective over $A$; see the proof of \cite[Theorem 3.1]{CI} (see also \cite[Chapter 9]{CFH}). For each $M\in A\text{\rm -Gproj}$, since $P$ is perfect, there exists $l>0$ such that $\Ext^{>l}_{A}(M, P)=0$, and hence $\Ext^{>l}_{A_\p}(M_\p,P_\p)=0$.  Set $d=\dim(R_\p)$. By the hypothesis, $\Gdim_R(M)<\infty$, and hence $\Gdim_{R_\p}(M_\p)<\infty$. It follows from \cite[Thoerem 4.13 and Corollary 4.15]{AB} that $\Gdim_{R_\p}(M_\p)\leq d$ and $\Ext^{>d}_{R_\p}(M_\p,R_\p)=0$. This implies $\Ext^{>d}_{A_\p}(M_\p,D_\p)=0$. Applying $\RHom_{A}(M, -)$ to (\ref{cones}), we get an exact triangle
$$
\RHom_A(M,K)[n{\rm -}1]\rightarrow \RHom_A(M,P)\rightarrow \RHom_A(M,D)\rightarrow \RHom_A(M,K)[n]
$$
in $\D(\Z)$. Combining this with $\Ext^{>l}_{A_\p}(M_\p,P_\p)=0=\Ext^{>d}_{A_\p}(M_\p,D_\p)$, we conclude that there exists an integer $m> {\rm max}\{l, d\} +n$ such that $\Ext^{>m}_{A_\p}(M_\p, K_\p)=0$ for each $M\in A\text{\rm -Gproj}$. By Lemma \ref{projective}, $K_\p$ is projective, and hence $D_\p$ is a perfect complex over $A_\p$  by (\ref{cones}). As $\p$ is arbitrary, it follows from \cite[Lemma 4.2]{IK22} that $D$ is a perfect complex over $A$. This completes the proof.
\end{proof}

\begin{remark}\label{question}
Keep the assumptions as Theorem \ref{observation}. 
The condition (\ref{basic}) is equivalent to that $\Gdim_{R_\p}(A_\p)<\infty$ and $\Hom_{R_\p}(A_\p,R_\p)$ is perfect over both $A_\p$ and $(A_{\p})^{\rm op}$ for each prime ideal $\p$ of $R$; see \cite[Corollary 6.3.4]{AIL} and \cite[Proposition III 6.6]{Bass}. By Theorem \ref{observation}, $\varphi\colon R\rightarrow A$ has ascent and descent of finite Gorenstein dimension property if and only if $\varphi_\p$ has ascent and descent of finite Gorenstein dimension property for each prime ideal $\p$ of $R$.
\end{remark}

In \cite[Section 4]{AF}, Avramov and Foxby raised a question:  Let $\varphi\colon R\rightarrow A$ be a finite ring homomorphism of commutative noetherian local rings. For a finitely generated $A$-module $M$, if both $\Gdim_R(A)$ and $\Gdim_A(M)$ are finite, is $\Gdim_R(M)$ finite?

By Theorem \ref{observation}, one can get the following result.

\begin{corollary}\label{application}
Let $\varphi\colon R\rightarrow A$ be a finite ring homomorphism of commutative noetherian rings. Assume $\RHom_R(A,R)$ is perfect over $A$. For a finitely generated $A$-module $M$, if  both $\Gdim_R(A)$ and $\Gdim_A(M)$ are finite, then $\Gdim_R(M)$ is finite.
\end{corollary}
\begin{remark}\label{descentapplication}
    When $\varphi$ is, in addition, a local ring homomorphism, Corollary \ref{application} is not new and can be proved by combining a result of Avramov and Foxby \cite[Lemma 6.5 and Theorem 7.11]{AF} with a result of Christensen \cite[Proposition 8.3]{Chr01}.
\end{remark}

A surjective ring homomorphism $\pi\colon R\rightarrow A$ of commutative noetherian rings is said to be a \emph{complete intersection} if the kernel of $\pi$ is generated by a regular sequence of $R$. Theorem \ref{observation} yields the following well-known result; see \cite[Theorem 2.3.12]{Chr00}.

\begin{corollary}\label{CImap}
Any complete intersection map has ascent and descent of finite Gorenstein dimension property.
\end{corollary}

\begin{proof}
Let $\pi\colon R\rightarrow A=R/(x_1,\ldots,x_n)$ be a complete intersection map, where $x_1,\ldots,x_n$ is a regular sequence of $R$. By \cite[Proposition 1.6.10 and Corollary 1.6.14]{BH}, the projective dimension of $A$ over $R$ is finite and $\RHom_R(A,R)\simeq A[-n]$. Then, it follows from Theorem \ref{observation} that $\pi$ has ascent and descent of finite Gorenstein dimension property.
\end{proof}

\section{Ascent and descent of Gorenstein projective property}\label{sectionadgp}

In this section, we give a characterization of the ascent and descent of Gorenstein projective property; see Theorem \ref{ADTF}. This characterization yields that the ascent and descent of Gorenstein projective property is a local property;
see Corollary \ref{local property}.

\begin{definition}
Let $\varphi\colon R\rightarrow A$ be a ring homomorphism between noetherian rings.
We say $\varphi$ has \emph{ascent and descent of Gorenstein projective property} if the following two conditions are satisfied$\colon$
\begin{enumerate}
\item For each finitely generated  $A$-module, it is Gorenstein projective over $A$  if and only if it is Gorenstein projective over $R$;

\item For each finitely generated $A^{\rm op}$-module, it is Gorenstein projective over $A^{\rm op}$  if and only if it is Gorenstein projective over $R^{\rm op}$.
\end{enumerate}
\end{definition}

The lemma below follows from  \cite[Theorem 3.1]{CI}; see also \cite[Proposition 9.1.17]{CFH}.
In the commutative case, it is a consequence of \cite[Corollary 2.3.8]{Chr00}.
 
\begin{lemma}\label{TFtest}
Let $A$ be a noetherian ring. If $M$ is a finitely generated  $A$-module with finite Gorenstein dimension and $\Ext^i_A(M,A)=0$ for all $i>0$, then $M$ is Gorenstein projective.
\end{lemma}


\begin{theorem}\label{ADTF}
Let $\varphi\colon R\rightarrow A$ be a ring homomorphism, where $R$ is a commutative noetherian ring and $A$ is a finite $R$-algebra.
The following two conditions are equivalent$\colon$
\begin{enumerate}
\item  $A$ is Gorenstein projective over $R$ and $\Hom_R(A,R)$ is projective over both $A$ and $A^{\rm op}$.

\item $\varphi$ has ascent and descent of Gorenstein projective property.
\end{enumerate}
Moreover, if $\varphi$ has ascent and descent of Gorenstein projective property, then for each finitely generated $A$-module $M$, $\Gdim_R(M)=\Gdim_A(M)$.
\end{theorem}

\begin{proof}
The equality $\Gdim_R(M)=\Gdim_A(M)$ is straightforward under the given assumption. We only need to prove the first statement.

$(1)\Rightarrow (2)$. It follows from Theorem \ref{observation} that $\varphi$ has ascent and descent of finite Gorenstein dimension property. It suffices to prove that for any finitely generated $A$-module $M$, $\Gdim_A(M)=0$ if and only if $\Gdim_R(M)=0$.
By an analogous argument, the assertion for finitely generated $A^{\rm op}$-modules also holds.

Assume $\Gdim_A(M)=0$. Since $\varphi$ has descent of finite Gorenstein dimension property, $\Gdim_R(M)<\infty$. Moreover, by the hypothesis that $\Hom_R(A,R)$ is a projective  $A$-module, we infer that for all $i>0$,
$$\Ext_R^i(M,R)\cong \Ext_A^i(M,\Hom_R(A,R)) = 0.$$
It follows immediately from Lemma \ref{TFtest} that $\Gdim_R(M)=0$.

Now, assume $\Gdim_R(M)=0$. Since $\Hom_R(A,R)$ is projective over $A^{\rm op}$ and $M$ is a Gorenstein projective $R$-module, we infer that $\Hom_R(A,R)\otimes_A M$ is also a Gorenstein projective $R$-module. Moreover, for any $i>0$ we conclude that
\begin{align*}
  \Ext^i_A(M, A)&\cong \Ext^i_A(M,\Hom_R(\Hom_R(A,R),R))\\
    &\cong \Ext_R^i(\Hom_R(A,R)\otimes_A M, R) = 0
\end{align*}
Note that $\Gdim_A(M)<\infty$ as $\varphi$ has ascent of finite Gorenstein dimension property. Then, it follows from Lemma \ref{TFtest} that $\Gdim_A(M)=0$. 

$(2)\Rightarrow (1)$. Since $\varphi$ has descent of Gorenstein projective property, $A$ is Gorenstein projective as an $R$-module. This yields that $\Hom_R(A,R)$ is also a Gorenstein projective $R$-module, and then $\Hom_R(A,R)$ is Gorenstein projective over both $A$ and $A^{\rm op}$ as $\varphi$ has ascent of Gorenstein projective property. There is a short exact sequence
$$0\rightarrow \Hom_R(A,R)\rightarrow F\rightarrow C\rightarrow 0$$
in $A\mbox{-mod}$, where $F$ is projective and $C$ is Gorenstein projective. Then the hypothesis yields that $C$ is also Gorenstein projective over $R$. We infer that the sequence is split from
$$\Ext^1_A(C,\Hom_R(A,R)) \cong  \Ext^1_R(C,R) = 0;$$
the isomorphism here follows from $\RHom_R(A,R)\simeq \Hom_R(A,R)$ and the adjunction $({\rm Res}, \RHom_R(A,-))$.
Hence, $\Hom_R(A,R)$ is a projective  $A$-module. The same argument will show that $\Hom_R(A,R)$ is a projective $A^{\rm op}$-module.
This completes the proof.
\end{proof}

From the above result and $(1)\Rightarrow (2)$ in Theorem \ref{observation}, there is a natural observation: if $\varphi$ has ascent and descent of Gorenstein projective property, then $\varphi$ has ascent and descent of finite Gorenstein dimension property.

Following \cite[Definition 3.1.18]{BH}, a commutative noetherian ring $A$ is said to be \emph{Gorenstein} provided that $A_\p$ is Iwanaga-Gorenstein for each prime ideal $\p$ of $A$.  A commutative Gorenstein ring need not be Iwanaga-Gorenstein. However, a commutative Gorenstein ring with finite Krull dimension must be Iwanaga-Gorenstein; see for example \cite[Theorem 4.1.1]{Buc}. The following is immediate from \cite[Theorem 3.3.7]{BH} and Theorem \ref{ADTF}.

\begin{corollary}
Let $\varphi\colon (R, \m)\rightarrow (A, \n)$ be a finite local homomorphism of Gorenstein local rings with the same Krull dimension. Then $\varphi$ has ascent and descent of Gorenstein projective property.
\end{corollary}

Next, we prove that the ascent and descent of Gorenstein projective property is a local property; see Corollary \ref{local property}. Before that, we need the following result, which can be proved by using the characterization of finitely generated Gorenstein projective modules; see \ref{gp}.

\begin{lemma}\label{local GP}
Let $R$ be a commutative noetherian ring and $A$ a finite $R$-algebra. For any finitely generated  $A$-module $M$, the following are equivalent$\colon$
\begin{enumerate}
\item $M$ is Gorenstein projective over $A$.

\item $M_\p$ is Gorenstein projective over $A_\p$ for each prime ideal $\p$ of $R$.

\item $M_\m$ is Gorenstein projective over $A_\m$ for each maximal ideal $\m$ of $R$.
\end{enumerate}
\end{lemma}

If we replace ``Gorenstein projective'' in Lemma \ref{local GP} with ``projective'', the above statement still holds.
Combining this with Theorem \ref{ADTF} and Lemma \ref{local GP}, we obtain that the ascent and descent of Gorenstein projective property is a local property.

\begin{corollary}\label{local property}
Let $\varphi\colon R\rightarrow A$ be a ring homomorphism, where $R$ is a commutative noetherian ring and $A$ is a finite $R$-algebra. The following are equivalent$\colon$
\begin{enumerate}
    \item\label{ad} $\varphi\colon R\rightarrow A$ has ascent and descent of Gorenstein projective property.

    \item\label{adp} $\varphi_\p\colon R_\p\rightarrow A_\p$ has ascent and descent of Gorenstein projective property for each prime ideal $\p$ of $R$.

      \item\label{adm} $\varphi_\m\colon R_\m\rightarrow A_\m$ has ascent and descent of Gorenstein projective property for each maximal ideal $\m$ of $R$.
\end{enumerate}
\end{corollary}

The notion of Frobenius extension of rings is a generalization of Frobenius algebra \cite{Kas}, which includes many interesting examples. Following \cite[Theorem 1.2]{Kad},  a ring homomorphism $S\rightarrow A$ is called a \emph{Frobenius extension} if the following equivalent conditions hold:

(1) $A$ is a finitely generated projective  $S$-module and there is an isomorphism $A\cong \Hom_{S}(A,S)$ over both $A$ and $S^{\rm op}$.

(2) $A$ is a finitely generated projective $S^{\rm op}$-module and there is an isomorphism $A\cong \Hom_{S^{\rm op}}(A,S)$ over both $S$ and $A^{\rm op}$.

Using Theorem \ref{ADTF}, one can immediately get the following. Indeed, a stronger result that any Frobenius extension has ascent and descent of Gorenstein projective property was established in \cite{Chen, Ren, Zhao}. However, as shown in Example \ref{converseFro}, the converse does not hold in general.

\begin{corollary}\label{motivation}
Let $\varphi\colon R\rightarrow A$ be a ring homomorphism, where $R$ is a commutative noetherian ring and $A$ is a finite $R$-algebra. If it is a Frobenius extension, then $\varphi$ has ascent and descent of Gorenstein projective property.
\end{corollary}

The next example shows that the converse of Corollary \ref{motivation} is not true.

\begin{example}\label{converseFro}
Consider the injection map which maps $t$ to $x^2$:
$$R=k\llbracket t\rrbracket/(t^2)\hookrightarrow A=k\llbracket x,y,z\rrbracket/(x^2-y^2,x^2-z^2,xy,xz,yz).$$
Both $R$ and $A$ are artinian Gorenstein local rings; see \cite[Example 3.2.11]{BH}. This map has ascent and descent of Gorenstein projective property, but it is not a Frobenius extension. Since $\dim_k(R)=2$ and $\dim_k(A)=5$, $A$ cannot be free over $R$. This implies that $A$ is not projective over $R$ as $R$ is local.
\end{example}

We end this section with an example that has ascent and descent of Gorenstein projective property, but it is not a Frobenius extension, not a ring homomorphism between Iwanaga-Gorenstein rings, and not a complete intersection map.

\begin{example}\label{eg2}
Let $S$ be a commutative noetherian ring which is not Iwanaga-Gorenstein. Consider the canonical surjection
$$\pi\colon R=S\llbracket x\rrbracket/(x^2)\twoheadrightarrow S.$$
One can check directly that $\Hom_R(S,R)\cong S$ as $S$-modules, and $S$ is Gorenstein projective as an $R$-module. It follows from Theorem \ref{ADTF} that $\pi$ has ascent and descent of Gorenstein projective property.
\end{example}

\section{Testing Iwanaga-Gorenstein rings}\label{testsection}

The main result of this section is the following.
\begin{theorem}\label{ascentG}
Let $\varphi\colon R\rightarrow A$ be a ring homomorphism, where $R$ is a commutative noetherian ring and $A$ is a finite $R$-algebra. Assume $\varphi$ has ascent and descent of finite Gorenstein dimension property. If $R$ is an Iwanaga-Gorenstein ring, then so is $A$. The converse holds if, in addition, the fibre $A\otimes_R R/\m$ is nonzero for each maximal ideal $\m$ of $R$.
\end{theorem}

\begin{proof}
Let $R$ be an Iwanaga-Gorenstein ring. Take a minimal injective resolution $ R\xrightarrow \simeq I$ over $R$.
Note that for each injective $R$-module $E$, $\Hom_R(A,E)$ is injective over both $A$ and $A^{\rm op}$. Combining with $\id_R(R)<\infty$, then $D\colonequals\Hom_R(A,I)$ is a bounded complex of $A\mbox{-}A$-bimodules and each term $D_i$ is injective over both $A$ and $A^{\rm op}$.

Since $\varphi$ has ascent and descent of finite Gorenstein dimension property, Theorem \ref{observation}  yields that $D$ is perfect over both $A$ and $A^{\rm op}$. Then one can choose a projective resolution $\pi\colon Q\xrightarrow \simeq D$ over $A^{\rm op}$ such that $Q$ is a bounded complex of finitely generated projective $A^{\rm op}$-modules.

For any finitely generated projective $A^{\rm op}$-module $P$, $\Hom_{A^{\rm op}}(P,A)$ is a projective $A$-module.
For any $A\mbox{-}A$-bimodule $E$, if $E$ is an injective $A$-module, then it follows from \cite[Theorem 3.2.16]{EJ} that as an $A$-module, $\Hom_{A^{\rm op}}(P,E)\cong E\otimes_A \Hom_{A^{\rm op}}(P,A)$ is also injective.
Hence, $\Hom_{A^{\rm op}}(Q,D)$ is a bounded complex of injective $A$-modules.

Consider the  $A$-linear morphisms
\begin{align*}
    A\xrightarrow \simeq \Hom_{A^{\rm op}}(D,D)
     \xrightarrow {\pi^\ast} \Hom_{A^{\rm op}}(Q,D),
\end{align*}
where the first quasi-isomorphism is from Proposition \ref{theta}. Since $D$ is homotopy injective, $\pi^\ast$ is also a quasi-isomorphism. It follows that $\id_A(A)<\infty$. The same argument will show $\id_{A^{\rm op}}(A)<\infty$. Therefore, $A$ is an Iwanaga-Gorenstein ring.

Conversely, we assume $A$ is an Iwanaga-Gorenstein ring and $A\otimes_R R/\m\neq 0$ for each maximal ideal $\m$ of $R$. For each $i$ and each  $A$-module $M$, there is an isomorphism
$$\Ext^i_R(M,R)\cong \Ext^i_A(M,D).$$
Note that $D$ is perfect over both $A$ and $A^{\rm op}$ by Theorem \ref{observation}. Combining this with $\id_A(A)<\infty$, we conclude that there exists an positive integer $j$ such that $\Ext^{>j}_R(M, R)=0$ for every  $A$-module $M$. For each maximal ideal $\m$ of $R$, by assumption $A\otimes_R R/\m$ is not zero, and then it is a direct sum of some copies of $R/\m$ over $R$. If we choose $M$ to be $A\otimes_R R/\m$, then $\Ext^{>j}_R(R/\m,R)=0$.
In particular, $$\Ext^{>j}_{R_\m}(R_\m/\m R_\m,R_\m)=0.$$
This implies that $\id_{R_\m}(R_\m)\leq j$; see \cite[Proposition 3.1.14]{BH}. Since $\m$ is arbitrary, $\id_R(R)\leq j$.
Hence, $R$ is an Iwanaga-Gorenstein ring. This completes the proof.
\end{proof}

Combining Theorem \ref{observation}, Theorem \ref{ascentG}, and Christensen’s result \cite[Proposition 8.3]{Chr01}, one can get the following result which is due to Avramov and Foxby \cite[4.4.4 and 7.7.2]{AF}; see also \cite[Theorem 6.2]{ISW}.

\begin{corollary}\label{recover}
Let $\varphi\colon (R,\m)\rightarrow (A,\n)$ be a finite local ring homomorphism. If $A$ is Gorenstein, then
$R$ is Gorenstein if and only if $\Gdim_R(A)$ is finite.
\end{corollary}

\begin{proof}
The ``only if'' part is clear since $R$ is a Gorenstein ring.

For the ``if'' part, assume $\Gdim_R(A)$ is finite. This implies that $D=\Hom_R(A,I)$ is in $\D^f_b(A)$ (see Lemma \ref{test}), where $I$ is an injective resolution of $R$. Since $A$ is Gorenstein, $\Gdim_A(D)<\infty$.  By Proposition \ref{theta}, $D$ is semi-dualizing. 
It follows from \cite[Proposition 8.3]{Chr01} that $D\simeq A[j]$ for some $j\in \mathbb Z$.  Combining this with the finiteness of $\Gdim_R(A)$, Theorem \ref{observation} yields that $\varphi$ has ascent and descent of finite Gorenstein dimension property. Thus, $R$ is Gorenstein by Theorem \ref{ascentG}.
\end{proof}

\begin{ack}
The authors are grateful to X.-W. Chen, Z.Y. Huang, S.B. Iyengar, and C. Psaroudakis for their helpful comments and suggestions. The authors sincerely thank an anonymous referee for carefully reading the manuscript and providing valuable comments and suggestions to improve the article. 
The first author is supported by the National Natural Science Foundation of China (No. 12401046) and the Fundamental Research Funds for the Central Universities (No. CCNU24JC001).
The second author is supported by the National Natural Science Foundation of China (No. 11871125). 
\end{ack}

\bibliography{}

\end{document}